\documentclass[12pt]{amsart} 
\usepackage{amssymb}
\usepackage[enableskew,vcentermath]{youngtab}

\newcommand{\g}{\mathfrak g} 
\newcommand{\ag}{\mathfrak a} 
\newcommand{\gl}{\mathfrak{gl}} 

\newcommand{\R}{\mathbb R}
\newcommand{\rtt}{\rightthreetimes} 
\newcommand{\csp}{\mathfrak{csp}}
\newcommand{\m}{\mathfrak{m}} 
\newcommand{\n}{\mathfrak{n}} 
\newcommand{\Gr}{\operatorname{Gr}}
\newcommand{\Stab}{\operatorname{Stab}}
\newcommand{\Der}{\operatorname{Der}}
\newcommand{\Hom}{\operatorname{Hom}}
\newcommand{\ad}{\operatorname{ad}}
\newcommand{\E}{\mathcal{E}}
\newcommand{\p}{\partial}
\newcommand{\dd}[1]{\frac{\p}{\p #1}}
\newcommand{\sll}{\mathfrak{sl}}
\newcommand{\Fol}{\operatorname{Fol}}

\newtheorem{thm}{Theorem}
\newtheorem{prop}{Proposition}
\newtheorem{lem}{Lemma}

\theoremstyle{definition}
\newtheorem{dfn}{Definition}
\newtheorem{ex}{Example}

\allowdisplaybreaks

\title{Symmetries of trivial systems of ODEs of mixed order} 
\author{Boris Doubrov \and Igor Zelenko}
\address{Belarussian State University, Nezavisimosti Ave.~4, Minsk 220050, Belarus;
 E-mail: doubrov@islc.org}
\address{Department of Mathematics, Texas A$\&$M University,
   College Station, TX 77843-3368, USA; E-mail: zelenko@math.tamu.edu}

\dedicatory{Dedicated to Mike Eastwood on the occasion of his 60th birthday}

\begin{document} 
\begin{abstract} We compute symmetry algebras of a system of two equations
$y^{(k)}(x)=z^{(l)}(x)=0$, where $2\le k < l$. It appears that there are many ways
to convert such system of ODEs to an exterior differential system. They lead to
different series of finite-dimensional symmetry algebras. For example, for $(k,l)=(2,3)$ 
we get two non-isomorphic symmetry algebras of the same dimension. We explore how these symmetry
algebras are related to both Sternberg prolongation of $G$-structures and Tanaka
prolongation of graded nilpotent Lie algebras.

Surprisingly, the case $(k,l)=(2,3)$ provides an example of a linear subalgebra $\ag$ in
$\gl(5,\R)$ such that the Sternberg prolongations of $\ag$ and $\ag^{t}$ are both of the
same dimension, but are non-isomorphic.

We also discuss the non-linear case and the link with flag structures on smooth manifolds.
\end{abstract}

\maketitle

\section{Introduction} The goal of this paper is to show that symmetry computations for
systems of ODEs of mixed order exhibit new phenomena not visible in case of systems of
ODEs of uniform order. It is sufficient to consider a system of two trivial ODEs of
different order to demonstrate these phenomena:
\begin{equation}\label{odekl}
y^{(k)}(x)=0, \quad z^{(l)}(x)=0,
\end{equation}
where $y(x)$ and $z(x)$ are two unknown functions. We shall always assume that $2\le k < l$,
as under these conditions the symmetry algebra of this system becomes finite-dimensional.

The first phenomenon is that mixed order systems admit different reformulations in
terms of exterior differential systems, which lead to different symmetry algebras. And
unlike inclusions between Lie algebras of point, contact and internal symmetries, there is
no inclusion between symmetry algebras coming from different EDS. For example,
we show that in the simplest non-trivial case of $(k,l)=(2,3)$ there are two different EDS's. 
Their symmetry algebras are both 15-dimensional, but one of them is isomorphic to 
$\gl(3,\R)\rtt S^2(\R^3)$, while another is isomorphic to $\csp(4,\R)\rtt \R^4$. 
For other pairs of $2\le k < l$ these symmetry algebras even have different dimensions.

The symmetries we consider in this paper are so-called \emph{external symmetries} 
(see~\cite{ako} for the terminology and the relationship between different types of symmetries). 
They preserve not only the internal geometry of the equation (which is just a 1-dimensional foliation and, 
thus, locally trivial), but also certain projections to jet spaces of lower order. And different
EDS interpretations of the same system of mixed order come from different projections.

The second phenomenon is that the same case $(k,l)=(2,3)$, interpreted in terms of
$G$-structures, leads to an explicit example of a linear subalgebra $\ag$ in $\gl(5,\R)$
such that the Sternberg prolongations of $\ag$ and $\ag^{t}$ are both of the same dimension,
but are non-isomorphic. In fact, we show that both symmetry algebras can be
obtained as total prolongations of certain $G$-structures related to the orbits of
$SL(2,\R)$-action on Grassmann varieties $\Gr_2(\R^{k+l})$ and $\Gr_{k+l-2}(\R^{k+l})$.

Finally, the third phenomenon is related to the use of Tanaka theory of graded nilpotent
Lie algebras and their prolongations for computing the symmetry algebras. It appears that
in one of the EDS reformulations one needs to consider graded nilpotent Lie algebras
$\m=\sum_{i<0}\m_i$, which are not generated by $\m_{-1}$. It appears that Tanaka theory
produces the expected result in this case as well, if we slightly modify the notion of
Tanaka prolongation.

The paper is organized as follows. In Section~\ref{sec:eds} we show how systems of mixed
order can be turned into exterior differential systems in two different ways. In
Section~\ref{sec:sym} we compute symmetry algebras for each of these exterior differential
systems and show that we get non-isomorphic symmetry algebras. In
Sections~\ref{sec:tan} and~\ref{sec:g} we show how these symmetry algebras appear
naturally as Tanaka prolongation of certain graded nilpotent Lie algebras and as Sternberg 
prolongation of subalgebras in $\gl(k+l,\R)$. In Section~\ref{sec:nlin} we discuss the case of non-linear systems of mixed order.  In Section~\ref{sec:flag} we link the geometry of non-linear systems of mixed order with so-called flag structures on smooth manifolds. Finally, in Section~\ref{sec:shifts} we describe other ways to associate exterior differential systems with equations of mixed order. 

\section{Two EDS interpretations} \label{sec:eds}

\subsection{The EDS of first kind}
Let $J^l(\R,\R^2)$ be the $l$-th jet space of maps from $\R$ to $\R^2$ with the coordinate system:
\[
(x,y,y_1,\dots,y_l,z,z_1,\dots,z_l).
\]
The system~\eqref{odekl} can be prolonged to the system of equations:
\[
y^{(k)}=y^{(k+1)}=\dots=y^{(l)}=0,\quad z^{(l)}=0,
\]
which can be considered as a submanifold $\E\subset J^l(\R,\R^2)$ defined by equations:
\[
y_{k}=y_{k+1}=\dots=y_{l}=0,\quad z_{l}=0.
\]
The remaining set of jet space coordinates $(x,y,z,y_1,\dots,y_{k-1},z_1,\dots,z_{l-1})$ forms 
a coordinate system on $\E$ itself. The system $\E$ defines also a one-dimensional vector bundle
on $\E$ tangent to lifts of all solutions of~\eqref{odekl}:
\begin{equation}\label{distE}
E = \left\langle \dd{x} + y_1\dd{y} + \dots + y_{k-1}\dd{y_{k-2}}+ z_1\dd{z} + 
\dots + z_{l-1}\dd{z_{l-2}}\right\rangle.
\end{equation}
Recall that the \emph{contact system} $C$ on $J^l(\R,\R^2)$ is defined by contact differential forms 
$dy_i-y_{i+1}dx$, $dz_i-z_{i+1}dx$, $i=0,\dots,l-1$, where $y_0=y$ and $z_0=z$. Note that the 
distribution $E$ can be defined as the intersection of  the contact system $C$ with $T\E$, or, in other
words, by the restriction of all contact forms on $\E$.

Due to the Lie-Backlund theorem this contact system canonically defines a sequence of integrable distributions transversal to $E$: 
\[
V_i = \langle dx, dy_r, dz_r \mid r = 0,\dots, i-1 \rangle^{\perp},
\]
where $i=1, \dots, l$. These distributions can also be defined as tangent spaces to fibers 
of the canonical projections $\pi_{l,i-1}\colon J^l(\R,\R^2)\to J^{i-1}(\R,\R^2)$. In particular, all
symmetries of the contact system $C$ are exactly the prolongations of vector fields from 
$J^0(\R,\R^2)=\R^3$. Such vector fields are called $\emph{point vector fields}$.

Symmetries (or, rather, \emph{infinitesimal symmetries}) of $\E$ are defined as point vector fields, 
which are at the same time tangent to the equation submanifold $\E$. Let $X$ be a restriction of 
such symmetry to $\E$ itself. Then it preserves both the vector distribution $E$ given by~\eqref{distE}
and all the intersections $V_i\cap T\E$. Denote these intersections by $F_i$. In particular, we have:
\[
F_1 = \left \langle \dd{y_i}, i=1,\dots,k-1; \dd{z_j},j=1,\dots,l-1\right\rangle.
\]
Simple computation shows that $E$ and $F_1$ alone allow to recover all other $F_i$, $i>1$, via:
\[
F_{i+1} = \{ Y \in F_i \mid [Y, E]\subset F_i \}.
\]
Note that if $l>k$, then the smallest of $F_i$ is
\[
F_{l-1} = \left\langle \dd{z_{l-1}} \right\rangle.
\]
And in this case the consequent brackets of $E$ with $F_{l-1}$ do not recover the complete sequence of $F_i$, as  the distribution $E+\sum_{i=0}^{\infty} \operatorname{ad}^iE(F_{l-1})$ is completely integrable and, for example, does not contain $F_{k-1}$.

This motivates the following
\begin{dfn}
The \emph{exterior differential system of first kind} associated with the system of ODEs~\eqref{odekl} is given by a pair of vector distributions
$E$ and $F_1$ on the equation manifold $\E\subset J^l(\R,\R^2)$.

\emph{Symmetries of first kind} are the vector fields $S$ on $\E$ that preserve both $E$ and $F_1$, that is $[S,E]\subset E$ and $[S,F_1]\subset F_1$. 
\end{dfn}

Although the symmetries are defined as vector fields on the equation manifold $\E$, they are actually external symmetries of the equation. While the distribution $E$ corresponds to the internal geometry of the equation, the second distribution $F_1$ comes from the projection $\pi\colon J^l(\R,\R^2)\to J^0(\R,\R^2)$. So, in terminology of~\cite{ako,olver} symmetries of first kind could be also called \emph{point symmetries}.

\subsection{The EDS of second kind}
Another way of defining the exterior differential system by~\eqref{odekl} is to start from the mixed jet space $J^{k,l}(\R,\R^2)$ with the local coordinate system $(x,y,y_1,\dots,y_k,z,z_1,\dots,z_l)$ and define the equation submanifold $\E$ by equations
\[
y_k=0, z_l=0.
\]
Then, as above, the coordinates  $(x,y,z,y_1,\dots,y_{k-1},z_1,\dots,z_{l-1})$ form a coordinate system on $\E$. Similar to the jet space $J^l(\R,\R^2)$ we define the contact system $C$ on $J^{k,l}(\R,\R^2)$ by the collection of contact forms
\begin{align*}
& dy_i - y_{i+1}dx, i= 0,\dots,k-1;
& dz_j - z_{j+1}dx, j= 0,\dots,l-1.
\end{align*}
The restrictions of these forms to the equation manifold $\E\subset J^{k,l}(\R,\R^2)$ define the same 
one-dimensional vector distribution~$E$ given by~\eqref{distE}. So, up to now everything looks very similar to the above.

Let $\pi\colon \E\to J^{k-1,l-1}(\R,\R^2)$ be the projection of $\E$ to the lower order jet space. The pull-back of the contact distribution from $J^{k-1,l-1}(\R,\R^2)$ to $\E$ is a 3-dimensional vector distribution $D$ on $\E$ defined by 1-forms:
\begin{align*}
& dy_i - y_{i+1}dx, i= 0,\dots,k-2;\\
& dz_j - z_{j+1}dx, j= 0,\dots,l-2.
\end{align*}
or, in terms of vector fields by:
\begin{align*}
X &= \dd{x} + y_1\dd{y} + \dots + y_{k-1}\dd{y_{k-2}}+ z_1\dd{z} + 
\dots + z_{l-1}\dd{z_{l-2}},\\
Y &= \dd{y_{k-1}},\\
Z &= \dd{z_{l-1}}.
\end{align*}

\begin{dfn}
The \emph{exterior differential system of second kind} associated with the system of ODEs~\eqref{odekl} is given by a pair of vector distributions $E \subset D$ on the equation manifold $\E\subset J^{k,l}(\R,\R^2)$.

\emph{Symmetries of second kind} are the vector fields $S$ on $\E$ that preserve both $D$ and $E$, that is $[S,D]\subset D$ and $[S,E]\subset E$. 
\end{dfn}

As in case of symmetries of first kind, symmetries of second kind are also external symmetries of the equation. The second distribution $D$ comes from the projection $\pi\colon J^{k,l}(\R,\R^2)\to J^{k-1,l-1}(\R,\R^2)$ and the contact structure on $J^{k-1,l-1}(\R,\R^2)$. So, in terminology of~\cite{ako,olver} symmetries of second kind could be also called \emph{contact symmetries}. However, we avoid this terminology, as in
our case there is to inclusion between Lie algebras of symmetries of first and second kind.

As in the case of differential systems of first kind, we can define a sequence of transversal completely integrable distributions in the following way. Let $F_{l-1}$ be the 2-dimensional completely integrable distribution generated by vector fields $Y$ and $Z$. As this is the only 2-dimensional completely integrable subdistribution of $D$, it is preserved by all symmetries of $D$. As $D=E\oplus F_{l-1}$, we see that the symmetries of second kind are exactly the vector fields preserving both $E$ and $F_{l-1}$.

Taking iterative brackets, we can further define:
\[
F_{i-1} = F_i + [E,F_i]\quad\text{for all } i \le l-1.
\]
These distributions are all transversal to $E$ and completely integrable.

\section{Symmetry computation} \label{sec:sym}

All symmetries of first kind are prolongations of the vector fields from $J^0(\R,\R^2)=\R^3$. See, for example,~\cite{olver} for the exact prolongation formulas.
\begin{prop}\label{propI}
The Lie algebra $\g^{I}_{k,l}$ of symmetries of first kind of equation~\eqref{odekl} is spanned over $\R$ by the following vector fields:

If $k=2$, $l>2$:
\begin{align}
& \dd{x}, \dd{y}, \label{a1}\\
& x\dd{x}, x\dd{y}, y\dd{x}, y\dd{y},z\dd{z}, \label{a2}\\
& x^2\dd{x}+xy\dd{y}+(l-1)xz\dd{z}, xy\dd{x} + y^2\dd{y} +(l-1)yz\dd{z},\label{a3}\\
& x^iy^j\dd{z}, 0\le i+j\le l-1.\label{a4}
\end{align}
Algebraically $\g^{I}_{2,l}$ is isomorphic to $\gl(3,\R)\rtt S^{l-1}(\R^3)$, where the subalgebra $\gl(3,\R)$ is spanned by~\eqref{a1}--\eqref{a3} and the commutative ideal $S^{l-1}(\R^3)$ is spanned by~\eqref{a4}.

\medskip
If $2<k<l$:
\begin{align}
& \dd{x}, x\dd{x}, x^2\dd{x}+(k-1)xy\dd{y}+(l-1)xz\dd{z}, y\dd{y}, z\dd{z},\label{b1}\\
& \dd{y}, x\dd{y}, \dots, x^{k-1}\dd{y},\label{b2}\\
& x^iy^j\dd{z}, 0\le i+(k-1)j\le l-1.\label{b3}
\end{align}
Algebraically $\g^{I}_{k,l}$ is isomorphic to $(\R\times\gl(2,\R)\rtt V_k)\rtt \left(\sum_{i=0}^{[l/k]} V_{l-ki}\right)$, where $V_r$ is an $r$-dimensional irreducible representation of $\gl(2,\R)$. The subalgebra $\R\times\gl(2,\R)$ is spanned by~\eqref{b1}, the subalgebra $V_k$ is spanned by~\eqref{b2}, and the commutative ideal $\sum_{i=0}^{[l/k]} V_{l-ki}$ is spanned by~\eqref{b3}.
\end{prop}
\begin{proof}
Take an arbitrary vector field $A\dd{x} + B\dd{y}+C\dd{z}$ on $\R^3$, prolong it to $J^l(\R,\R^2)$ and check that it preserves the equation. This results in a system of PDEs on the functions $A,B,C$, which are easy to solve. This results in the above Lie algebras of point symmetries.
\end{proof}

To describe the symmetries of the second kind we need the following technical result, which is also of its own interest. Let $X$ be the operator of total derivative on $J^{r+1}(\R,\R)$ with the standard coordinate system $(x,z_0,\dots,z_{r+1})$:
\[
X = \dd{x} + z_1\dd{z_0} + \dots + z_r \dd{z_{r-1}}+\dots +z_{r+1} \dd{z_r}.
\]
Denote by $g_{i,j}$ the following function on $J^{r+1}(\R,\R)$:
\[
g_{i,j} = \frac{x^{i+j}}{(i+j)!} X^i(z_0/x^j), \quad i,j\ge 0.
\]
Let $g_{i,j}^{(s)}=\frac{\partial^{s} g_{i,j}}{\partial x^s}$ for any $s\ge 0$. In particular, we see that $g_{i,j}^{(0)}=g_{i,j}$ and $g_{i,j}^{(s)}=0$ for $s>i$.

It is easy to see that each $g_{i,j}^{(s)}$ is linear in $z_{s},z_{s+1},\dots,z_i$, polynomial in $x$ and having a constant coefficient at $z_{s}$. In particular, for fixed $i,j$ all functions $g_{i,j}^{(s)}$, $s=0,\dots,i$ are algebraically independent (even over $\R[x]$). 
\begin{lem}\label{lem:g}
\begin{enumerate}
\item[(a)] The following identity holds:
\begin{equation}\label{eq:ij}
g_{i,j} - (x/i)g^{(1)}_{i,j} + j g_{i-1,j+1} = 0,\quad\text{for all } i,j>0.
\end{equation}
\item[(b)] The space of solutions for the system of linear equations:
\begin{equation}\label{eq:d2}
\begin{aligned}
X^2 ( f ) &= 0;\\
\frac{\partial f}{\partial z_{r+1}} &= 0,
\end{aligned}
\end{equation}
is $r+3$-dimensional and is spanned by the functions $1$, $x$ and $g_{r,2}^{(s)}$, $s=0,\dots,r$. 
\end{enumerate}
\item[(c)] The space of solutions for the following system of linear equations:
\begin{equation}\label{eq:dk}
\begin{aligned}
X^{p+1} ( f ) &= 0, \quad p\ge 1;\\
\frac{\partial f}{\partial z_{r+1}} &= 0,
\end{aligned}
\end{equation}
is spanned by $W^p=\{ f_1\dots f_{p} \mid f_i\in W\}$, where $W$ is the above $(r+3)$-dimensional space of solutions of~\eqref{eq:d2}.
\item[(d)] The space of solutions for the following system of linear equations:
\begin{equation}\label{eq:dkl}
\begin{aligned}
X^{p+1} ( f ) &= 0, \quad p\ge 1;\\
\frac{\partial f}{\partial z_{r+1-q}} & =\dots =\frac{\partial f}{\partial z_{r+1}} = 0,\quad q\ge 0;
\end{aligned}
\end{equation}
is spanned by $x^i g_{r-q,q+2}^{(s_1)} \dots g_{r-q,q+2}^{(s_j)}$, where $i+(q+1)j\le p$.
\end{lem}
\begin{proof}
(a) Easily follows from the fact that $X$ commutes with $\dd{x}$ and  from the identity:
\[
X^i(z_0/x^j) = X^i(x\cdot z_0/x^{j+1}) = x X^i(z_0/x^{j+1}) + i X^{i-1}(z_0/x^{j+1}).
\]

(b) It is clear that $1$ and $x$ are solutions of system~\eqref{eq:d2}. Further, we have:
\begin{multline*}
X^2(g_{r,2}) =\tfrac{x^r}{r!} X^r(z_0/x^2) + 2\tfrac{x^{r+1}}{(r+1)!}X^{r+1}(z_0/x^2) + \tfrac{x^{r+2}}{(r+2)!}X^{r+2}(z_0/x^2) \\
= \tfrac{x^r}{r!} \sum_{i=0}^r \tbinom{r}{i} X^{r-i}(x^{-2})X^i(z_0) + 2\tfrac{x^{r+1}}{(r+1)!}  \sum_{i=0}^{r+1} \tbinom{r+1}{i} X^{r+1-i}(x^{-2})X^i(z_0) + \\ + \tfrac{x^{r+2}}{(r+2)!} \sum_{i=0}^{r+1} \tbinom{r+2}{i} X^{r+2-i}(x^{-2})X^i(z_0)\\ 
= \tfrac{x^r}{r!} \sum_{i=0}^{r+1} z_{i}\left(\tbinom{r}{i} X^{r-i}(x^{-2})+\tfrac{2x}{r+1}\tbinom{r+1}{i}X^{r+1-i}(x^{-2})+\tfrac{x^2}{(r+1)(r+2)}\tbinom{r+2}{i}X^{r+2-i}(x^{-2})\right)\\
= x^r\sum_{i=0}^{r+1} \tfrac{z_i}{i!} \left(\tfrac{1}{(r-i)!}X^{r-i}(x^{-2})+\tfrac{2x}{(r+1-i)!}X^{r+1-i}(x^{-2})+\tfrac{x^2}{(r+2-i)!}X^{r+2-i}(x^{-2})\right)\\
= x^r\sum_{i=0}^{r+1} (-1)^{r-i}\tfrac{z_i}{i!} \left(\tfrac{1}{(r-i)!}\tfrac{(r+1-i)!}{x^{r-i+2}}-\tfrac{2x}{(r+1-i)!}\tfrac{(r+2-i)!}{x^{r+3-i}}+\tfrac{x^2}{(r+2-i)!}\tfrac{(r+3-i)!}{x^{r+4-i}}\right)\\
= \sum_{i=0}^{r+1} (-1)^{r-i}\frac{z_i x^{i-2}}{i!} ((r+1-i) - 2(r+2-i) + (r+3-i)) = 0.
\end{multline*}
(Here the terms which do not make sense such as $X^{-1}(x^{-2})$ are assumed to be $0$.) As operators $X$ and $\dd{x}$ commute, we see that all functions $g_{r,2}^{(k)}$, $k=0,\dots,r$, indeed satisfy system~\eqref{eq:d2}. 

Denote by $Z_i$, $i=0,\dots,r+1$, the differential operator $\dd{z_{i}}$. Note that $[Z_i, X] = Z_{i-1}$, or what is the same $X Z_i = Z_i X - Z_{i-1}$ for all $i=1,\dots,r+1$. Taking $i=r+1$, we get that
\[
(Z_{r+1} X - Z_r)f = 0 
\] 
for any solution $f$ of system~\eqref{eq:d2}. Multiplying this identity by $X$ from the left and using bracket relations between $X$ and $Z_i$, we get
\[
X(Z_{r+1} X - Z_r)f = (Z_{r+1}X^2-2Z_r X + Z_{r-1})f= (-2Z_r X + Z_{r-1})f = 0.
\]
Proceeding in a similar way we get that:
\[
((r+1-i)Z_{i+1} X - Z_i)f=0\quad\text{for all }i=0,\dots,r.
\]
and also $Z_0Xf=XZ_0f=0$. Multiplying these identities by $Z_j$, $j=0,\dots,r+1$ on the left and taking differences we get:
\begin{align*}
& Z_i Z_0 f  =0,\\ 
& ((r+1-j)Z_iZ_{j+1}-(r+1-i)Z_jZ_{i+1})f = 0,
\end{align*}
for all $i,j=0,\dots,r$. This generates all quadratic relations $Z_iZ_j(f)=0$ for $i,j=0,\dots,r$. Thus, $f$ is linear with respect to $z_0,\dots,z_r$. Then easy computation shows that the solution space of system~\eqref{eq:d2} for has dimension $r+3$.

(c) It is easy to see that all elements of $W^p$ satisfy system~\eqref{eq:dk} for arbitrary $p\ge 2$. The proof that the solution space of system~\eqref{eq:dk} for any $p\ge 2$ coincides with $W^p$ follows from the dimension count similar to the case $p=1$.

(d) Solutions of system~\eqref{eq:dkl} are exactly the elements of the space $W^p$ that do not depend on $z_{r+1-i}$, $i=0,\dots,q$. Identity~\eqref{eq:ij} for $i=r$, $j=2$ implies that $g_{r-1,3}$ does not depend on $z_{r},z_{r+1}$ and lies in $W^2$. Similarly, by induction we get that $g_{r-q,q+2}$ does not depend on $z_{r+1-i}$, $i=0,\dots,q$, and lies in $W^{q+1}$. Clearly, all functions $g_{r-q,q+2}^{(s)}$, $s=0,\dots,r-q$, also have this property. This implies that all products $x^i g_{r-q,q+2}^{(s_1)} \dots g_{r-q,q+2}^{(s_j)}$, with $i+(q+1)j\le p$ are (linearly independent) solutions of~\eqref{eq:dkl}. The dimension count similar to~(b) proves that this is the complete solution space of system~\eqref{eq:dkl}.
\end{proof}

The symmetries of second kind are not always prolongations of vector fields on $\R^3$, but they can still be given as prolongations of vector fields from mixed jet space $J^{0,l-k}(\R,\R^2)$ with the standard coordinate system $(x,y,z,z_1,\dots,z_{l-k})$. The vector fields from $J^{0,0}(\R,\R^2)=\R^3$ can be naturally prolonged to $J^{0,l-k}(\R,\R^2)$ using the standard prolongation formulas for variables $z_i$ alone.

\begin{prop}\label{propII}
The Lie algebra $\g^{II}_{k,l}$ of symmetries of second kind is spanned over $\R$ by the following vector fields:

If $k=2,l=3$:
\begin{align*}
& (x^2z_1/2 -zx)\dd{x}+(xyz_1/2-yz)\dd{y}+(x^2z_1^2/4-z^2)\dd{z}+(xz_1^2/2-zz_1)\dd{z_1},\\
& 2(xz_1-z)\dd{x}+yz_1\dd{y}+xz_1^2\dd{z}+z_1^2\dd{z_1},\quad z_1\dd{x}+z_1^2/2\dd{z},\\
& x^2 \dd{x} +xy\dd{y}+2xz\dd{z}+2z\dd{z_1},\\
& x\dd{x}-z_1\dd{z_1},\quad z\dd{z}+z_1\dd{z_1},\\
& x^2 \dd{z}+2x\dd{z_1},\quad  x\dd{z}+\dd{z_1},\\
& \dd{x},\dd{z},\quad  y\dd{y},\\
& \quad (xz_1-2z)\dd{y}, x\dd{y}, z_1\dd{y}, \dd{y}.
\end{align*}
Algebraically $\g^{II}_{2,3}$ is isomorphic to $\csp(4,\R)\rtt \R^4$, where the subalgebra $\csp(4,\R)$ is spanned by vector fields in the first 6 lines and the commutative ideal $\R^4$ is spanned by the vector fields in the last line.

\medskip
If $2\le k<l$ and $(k,l)\ne (2,3)$:
\begin{align}
& \dd{x}, x\dd{x}, x^2\dd{x}+(k-1)xy\dd{y}+(l-1)xz\dd{z}, y\dd{y}, z\dd{z},\label{d1}\\
& \dd{z}, x\dd{z},\dots,x^{l-1}\dd{z},\label{d2}\\
& \dd{y}, x\dd{y},\dots,x^{k-1}\dd{y},\label{d3}\\
& g_{l-k,k}^{(s)}\dd{y}, 0\le s \le l-k, \label{d4}
\end{align}
Here vector fields~\eqref{d1} and~\eqref{d2} are assumed to be prolonged to $J^{0,l-k}(\R,\R^2)$. 

Algebraically $\g^{II}_{k,l}$ is isomorphic to $(\R\times \gl(2,\R)\rtt W)\rtt {V_k\oplus V_l}$, where $\R\times \gl(2,\R)$ is spanned by~\eqref{d1}, $V_l$ is spanned by \eqref{d2}, $V_k$ is spanned by~\eqref{d3}, and $W$ is an $(l-k+1)$-dimensional space spanned by~\eqref{d4}, which can be identified with an irreducible submodule of $\Hom(V_l, V_k)\subset \gl(V_k\oplus V_l)$.
\end{prop}
\begin{proof}
Any symmetry of second kind is a prolongation of the vector field 
\[
S = A\dd{x}+B\dd{y}+C_0\dd{z}+C_1\dd{z_1}+\dots+C_{l-k}\dd{z_{l-k}}
\] 
from the mixed order jet space $J^{0,l-k}(\R,\R^2)$. As it preserves contact forms $dz_i-z_{i+1}d{x}$, $i=0,\dots,l-k-1$, it is easy to see that it projects to a contact vector field on $J^{l-k}(\R,\R)$:
\[
\overline S = A \dd{x} + C_0\dd{z_0} + \dots + C_{l-k}\dd{z_{l-k}}.
\]
Moreover, simple computation shows that $\overline S$  is contact symmetry of the equation $z_l=0$, while the kernel of the projection $S\mapsto \overline S$ consists of vector fields:
\[
\big(cy + f(x,z_0,\dots,z_{l-k})\big)\dd{y},
\]
where the function $f$ satisfies system~\eqref{eq:dkl} for $r=l-2$, $p=k-1$ and $q=k-2$.

Direct computation shows that all contact  symmetries of the equation $z_l=0$ can be extended to the symmetries of the EDS of the second kind. In particular, in case $l=3$ the contact symmetry algebra is isomorphic to $\mathfrak{sp}(4,\R)$, and it is embedded as a subalgebra into $\g^{II}$. 

Further, item (d) of Lemma~\ref{lem:g} implies that $f$ is a linear combination of products $x^i g^{(s_1)}_{l-k,k}\dots g^{(s_j)}_{l-k,k}$, where $i+(k-1)j\le k-1$. Thus, either $j=0$, $i=0,\dots,k-1$ or $j=1$, $i=0$. This completes the proof of the theorem.
\end{proof}

\section{Symmetry algebras via Tanaka prolongation} \label{sec:tan}
In this section we show how symmetry algebras of first and second kind can be obtained as Tanaka prolongations of certain graded nilpotent Lie algebras.  

Let us fix the following frame on the equation $\E\subset J^{k,l}(\R,\R^2)$:
\begin{align*}
Y_i &= \dd{y_i},\quad i = 0,\dots, k-1;\\
Z_j &= \dd{z_j},\quad j = 0,\dots, l-1;\\
X &= \dd{x} + y_1\dd{y_0} + \dots + y_{k-1}\dd{y_{k-2}} + z_1\dd{z_0} + \dots + z_{l-1}\dd{z_{l-2}}.
\end{align*}
It is easy to see that these vector fields form a basis of a nilpotent Lie algebra $\n$ with the only non-zero Lie brackets being:
\[
[Y_i, X] = Y_{i-1}, \ [Z_j, X] = Z_{j-1}, \quad i,j\ge 1.
\]
Let us introduce two different gradings on $\n$. Both gradings are concentrated in negative degree and have $\deg X = -1$. The first grading is defined by $\deg Z_i = \deg Y_i = i-l$. In particular, $\deg Z_{l-1} = -1$ and $\deg Y_{k-1} = -1 + (k-l)$. The second grading is defined by $\deg Y_{i} = i-k, \deg Z_{j} = j-l$.  In particular, $\deg Y_{k-1} = \deg Z_{l-1} = -1$. To distinguish these two cases we shall denote the Lie algebra $\n$ equipped with a first grading as $\n^{I}$ and with a second grading as $\n^{II}$.

Let us recall the notion of Tanaka prolongation of graded nilpotent Lie algebras. Let $\m$ be an arbitrary negatively graded nilpotent Lie algebra of depth $\mu$, that is $\m=\sum_{i=1}^\mu \m_{-i}$. We recall that $\m$ is called \emph{fundamental}, if $\m$ is generated by $\m_{-1}$. 

For the above graded nilpotent Lie algebras we have:
\begin{align*}
\n^{I}_{-1} &= \langle X, Z_{l-1} \rangle\quad \text{for }  k < l;\\
\n^{I}_{-1} &= \langle X, Z_{l-1}, Y_{k-1} \rangle\quad \text{for }  k=l;\\
\n^{II}_{-1} &= \langle X, Z_{l-1}, Y_{k-1} \rangle\quad \text{for all }  k\le l.
\end{align*}
Thus, we see that $\n^{I}$ is \textbf{not fundamental} in case of $k<l$.

The \emph{universal  (Tanaka) prolongation of $\m$ } is defined as a largest graded Lie algebra $\g(\m)$ satisfying the following two conditions:
\begin{enumerate}
\item $\g_{i}(\m) = \m_{i}$ for all $i<0$;
\item for any $A\in \g_i(\m)$, $i\ge 0$, the equality $[A,\g_{-}(\m)]=0$ implies $A=0$.
\end{enumerate}
If $\m$ is fundamental, then the second condition can be replaced by:
\begin{enumerate}
\item[(2')] for any $A\in \g_i(\m)$, $i\ge 0$, the equality $[A,\g_{-1}(\m)]=0$ implies $A=0$.
\end{enumerate}

The universal Tanaka prolongation has a natural geometric sense in terms of symmetries of left-invariant EDS's on Lie groups. Namely, let $M$ be a Lie group with the Lie algebra $\m$. Define the sequence $T^{-k}M$ of left-invariant vector distributions on $M$ by the condition:
\[
T^{-k}_eM = \sum_{i=1}^k \m_{-i}.
\]
In extreme cases we have $T^0M=0$ and $T^{-\mu}M=TM$. If $\m$ is fundamental, then the complete sequence is defined by $T^{-1}M$.

We call the flag $\{ T^{-k} M\}$ \emph{the standard flag of type $\m$}. An infinitesimal symmetry of flag $\{ T^{-k} M\}$ is a vector field $X$ on $M$ such that $[X, T^{-i}M]\subset T^{-i}M$ for all $i=1,\dots,\mu$. Denote by $\bar\g(\m)$ the Lie algebra of all germs of infinitesimal symmetries at the identity of $M$. It is easy to show that this Lie algebra is well-defined. In general, it can be infinite-dimensional. But as it contains all germs of right-invariant vector fields on $M$, it is transitive at the identity $e$ (i.e., the values of all elements from $\bar \g (\m)$ at $e$ span all tangent space $T_eM$). 

This Lie algebra $\bar\g(\m)$ can be equipped with a natural decreasing filtration by setting:
\[
\bar\g^{-i}(\m) = \{ X\in \bar\g(\m) \mid X_e \subset T^{-i}_eM \},\quad\text{for all } i\ge 0.
\]
and extending it in the positive direction as follows:
\[
\bar\g^{i}(\m) = \{ X\in \bar\g(\m) \mid [X, \bar\g^{-j}(\m)]\subset \bar\g^{i-j}(\m) \forall j>0\},\quad \text{for } i > 0.
\]
In particular, $\bar\g^{-\mu}(\m)=\bar \g(\m)$ and $\bar\g^{0}(\m)$ is a subalgebra of all germs of infinitesimal symmetries that vanish at $e$. 

Finally, we define $\g(\m)$ as the graded Lie algebra associated with the filtered Lie algebra $\bar\g(\m)$:
\[
\g_i(\m) = \bar\g^{i}(\m)/\bar\g^{i+1}(\m),\quad\text{for all } i\in\mathbb{Z}.
\]

The fundamental result of N.~Tanaka and K.~Yamaguchi says: 
\begin{thm}[\cite{tan1,yama}]
\label{thm:yama}
The graded Lie algebra $\g(\m)$ coincides with the universal Tanaka prolongation of the graded nilpotent Lie algebra $\m$. If, moreover, $\g(\m)$ (or $\bar \g(\m)$) is finite-dimensional, then $\bar \g(\m)$ is isomorphic to $\g(\m)$ as filtered Lie algebras.
\end{thm}
We note that Tanaka and Yamaguchi proved this result only in the case when $\m$ is fundamental, i.e.~is generated by $\m_{-1}$. But their proof works only without modification in the case of arbitrary graded nilpotent Lie algebras assuming we define the Tanaka prolongation as above.

This result can be generalized to the case when we additionally put extra linear restrictions at degree $0$. Note that $\g_0(\m)$ is exactly the Lie algebra $\Der_0(\m)$ of all degree preserving derivations of $\m$. Let $\g_0$ be an arbitrary subalgebra in $\Der_0(\m)$. Then we can define the \emph{universal prolongation of the pair $(\m,\g_0)$} as the largest graded Lie algebra $\g(\m,\g_0)$ satisfying the following conditions:
\begin{enumerate}
\item $\g_{i}(\m,\g_0) = \m_{i}$ for all $i<0$ and $\g_{0}(\m,\g_0) = \g_0$;
\item for any $X\in \g_i(\m,\g_0)$, $i\ge 0$, the equality $[X,\g_{-}(\m)]=0$ implies $X=0$.
\end{enumerate}

Now suppose for simplicity that $\g_0$ is defined as a subalgebra of $\Der_0(\m)$ that stabilizes a family of graded subspaces $\{E_i\}_{i\in I}$ in $\m$. We shall write this as $\g_0=\Stab(\{E_i\}_{i\in I})$. Then the Lie algebra $\g(\m,\g_0)$ can also be interpreted in terms of symmetries of a family of left-invariant vector distributions on a Lie group $M$ with Lie algebra $\m$. Namely, subspaces $E_i$, $i\in I$, extend to left-invariant vector distributions on $M$, and we define a
filtered Lie algebra $\bar \g(\m,\g_0)$ as symmetries of these distributions together with the distributions
$T^{-i}M$ defined above. The proof of Theorem~\ref{thm:yama} stays valid in this generalized case as well: the graded Lie algebra $\g(\m,\g_0)$ coincides with the universal Tanaka prolongation of the pair $(\m,\g_0)$.

We can now formulate the following result. 
\begin{thm}
The Lie algebras $\g^{I}_{k,l}$ and $\g^{II}_{k,l}$ from Section~\ref{sec:sym} coincide with the Tanaka prolongations of the pairs $(\n^{I}_{k,l},\Stab(E))$ and $(\n^{II}_{k,l},\Stab(E))$ where $E=\langle X \rangle$.
\end{thm}
\begin{proof}
Directly follows from the simple observation that both EDS's of first and second kind associated with the system~\eqref{odekl} coincide with left-invariant distributions defined by Lie algebras $\n^{I}_{k,l}$ (first kind) $\n^{II}_{k,l}$ (second kind) and the transversal left-invariant distribution corresponding to the subspace $E$.
\end{proof}

\begin{ex}
Consider the case $(k,l)=(2,3)$. According to the above theorem, the universal Tanaka prolongation is isomorphic to the Lie algebra $\g^{I}_{2,3}$. Let us describe the corresponding grading of $\g^{I}_{2,3}$. We have:
\begin{align*}
&\deg -3: && \langle \dd{y}, \dd{z}\rangle;\\
&\deg -2: && \langle x\dd{y}, x\dd{z}\rangle;\\
&\deg -1: && \langle x^2\dd{z}, \dd{x}\rangle;\\
&\deg 0: && \langle x\dd{x}, y\dd{y}, z\dd{z}, y\dd{z}\rangle;\\
&\deg 1: && \langle x^2\dd{x}+xy\dd{y}+2xz\dd{z}, xy\dd{z} \rangle;\\
&\deg 2: && \langle y\dd{x} \rangle;\\
&\deg 3: && \langle xy\dd{x}+y^2\dd{y}+2yz\dd{z}, y^2\dd{z} \rangle.
\end{align*}

More generally, for arbitrary $2\le k < l$ the grading of $\g^{I}_{k,l}$ can be determined by setting $\deg x = 1$, $\deg y=\deg z = l$ in formulas of Proposition~\ref{propI}.
\end{ex}

\begin{ex}
The grading of $\g^{II}_{2,3}$ can be determined by setting $\deg x = 1$, $\deg y=\deg z_1=2$, $\deg z = 3$ in Proposition~\ref{propII}:
\begin{align*}
&\deg -3: && \langle \dd{z}\rangle;\\
&\deg -2: && \langle \dd{y}, x\dd{z}+\dd{z_1}\rangle;\\
&\deg -1: && \langle \dd{x}, x\dd{y}, x^2\dd{z}+2x\dd{z_1}\rangle;\\
&\deg 0: && \langle x\dd{x}-z_1\dd{z_1}, z\dd{z}+z_1\dd{z_1}, y\dd{y}, z_1\dd{y}\rangle;\\
&\deg 1: && \langle x^2 \dd{x} +xy\dd{y}+2xz\dd{z}+2z\dd{z_1}, z_1\dd{x}+z_1^2/2\dd{z},  (xz_1-2z)\dd{y} \rangle;\\
&\deg 2: && \langle 2(xz_1-z)\dd{x}+yz_1\dd{y}+xz_1^2\dd{z}+z_1^2\dd{z_1} \rangle;\\
&\deg 3: && \langle (x^2z_1/2 -zx)\dd{x}+(xyz_1/2-yz)\dd{y}+(x^2z_1^2/4-z^2)\dd{z}+(xz_1^2/2-zz_1)\dd{z_1}\rangle.
\end{align*}

Similarly, for all other $(k,l)$ the grading of $\g^{II}_{k,l}$ is defined by $\deg x = 1$, $\deg y=k$, $\deg z_i = l-i$, $i=0,\dots,l-k$.
\end{ex}

\section{Symmetry algebras via Sternberg prolongation} \label{sec:g}
In this section we show how symmetry algebras $\g^{I}_{k,l}$ and $\g^{II}_{k,l}$ can be obtained as Sternberg prolongations of certain subalgebras in $\gl(k+l,\R)$. These subalgebras are exactly the symmetry algebras of homogeneous rational curves in $\Gr(k+l-2,k+l)$ and $\Gr(2,k+l)$ in the cases of symmetries of first and second kind respectively.

Let us recall the notion of Sternberg prolongation of a linear Lie algebra. Let $V$ be an arbitrary finite-dimensional vector space and let $\ag\subset \gl(V)$ be a linear Lie algebra. Denote also by $D(V)$ the Lie algebra of polynomial vector fields on $V$:
\[
D(V) = \sum_{i=0}^{\infty} S^i(V^*)\otimes V.
\]
It is a graded Lie algebra, where $D_k(V)=S^{k+1}(V^*)\otimes V$, so that $D_{-1}(V)=V$ and $D_0(V)$ is identified with $\gl(V)$. By the Sternberg prolongation of $\ag\subset \gl(V)$ we understand a graded  subalgebra $\g$ of $D(V)$ that can be defined via three equivalent ways:
\begin{enumerate}
\item $\g$ is a largest graded subalgebra of $D(V)$ such that $\g_{-1}=D_{-1}(V)=V$ and $\g_0=\ag$;
\item $\g_{-1}=D_{-1}(V)$, $\g_0=\ag$ and $\g_{i+1}=\{u\in D_{i+1}(V)\mid [u,\g_{-1}]\subset \g_i\}$ for all $i\ge 0$;
\item $\g_{i} = S^{i+1}(V^*)\otimes V \cap S^{i}(V^*)\otimes \ag$ for all $i\ge -1$.
\end{enumerate}

Let $V$ be now the $(k+l)$-dimensional vector space with the basis $\{ e_0, e_1,\dots, e_{k-1}, f_0, f_1, \dots, f_{l-1} \}$. Define the nilpotent linear operator $X\colon V\to V$ by:
\begin{align}\label{Xdef}
X(e_i)&=e_{i-1}, \quad i=1,\dots,k-1;\quad X(e_0)=0;\\
X(f_i)&=f_{i-1}, \quad i=1,\dots,l-1;\quad X(f_0)=0 \nonumber.
\end{align}
Denote by $\exp(tX)$ the corresponding one-parameter group in $GL(V)$. 

We shall consider two different gradings of $V$:
\begin{enumerate} 
\item the grading of first kind: $\deg e_i = \deg f_i = i-l$;
\item the grading of second kind: $\deg e_i = i-k$, $\deg f_j = j-l$.
\end{enumerate}
Note that for both gradings the operator $X$ has degree $-1$. But they induce different gradings on the Lie algebra $\gl(V)$. And for both gradings all elements of $V$ are concentrated in negative degree.

Define a curve $\gamma^{I}$ in $\Gr(k+l-2,k+l)$ as the closure of the orbit of the one-parameter subgroup $\exp(tX)$ through the codimension two subspace $V_1=\langle e_1, \dots, e_{k-1}, f_1,\dots, f_{l-1}\rangle$. Similarly, define a curve $\gamma^{II}$ in $\Gr(2,k+l)$ as a closure of the orbit of $\exp(tX)$ through the two-dimensional subspace $V_2=\langle  e_{k-1}, f_{l-1}\rangle$. The curve $\gamma^{II}$ is well-known in projective geometry as a rational normal scroll $S_{k,l}$ (see~\cite{harris}), while $\gamma^{I}$ is its dual curve. Let $A^{I}_{k,l}, A^{II}_{k,l}\subset GL(V)$ be the symmetry groups of curves $\gamma^{I}$ and $\gamma^{II}$ respectively.  Denote by $\ag^{I}_{k,l}$ and $\ag^{II}_{k,l}$ the corresponding subalgebras in $\gl(V)$. 

These symmetry algebras can be easily computed in a purely algebraic way as follows:
\begin{prop}[\cite{dou:kom,dou:05}]
\label{flagprolong}
The subalgebras $\ag^{I}_{k,l}, \ag^{II}_{k,l}$ are the largest graded subalgebras of $\gl(V)$ whose negative part is one-dimensional and is generated by $X$. They can be constructed inductively as:
\begin{align*}
\ag_{-1} &= \langle X \rangle; \\
\ag_{i} & = \{ u\in \gl_{i}(V) \mid [u, X]\subset \ag_{i-1} \},\quad\text{for }i\ge 0,
\end{align*}
where $\ag$ is either $\ag^{I}_{k,l}$ or $\ag^{II}_{k,l}$ and $V$ is equipped with the grading of first or second kind respectively.
\end{prop}

In fact, symmetry algebras of rational normal scrolls are well-known. If $k<l$, then
\[
\ag^{II}_{k,l} = \left\{\left. 
\begin{pmatrix} \rho_k(A)+ e_1 E_{k} & B \\ 
0 & \rho_l(A) +  e_2 E_{l} \end{pmatrix} 
\,\right|\, \begin{matrix} A\in \sll(2,\R), e_1,e_2\in\R, \\ B \subset V(l-k)  \end{matrix}
\right\},
\]
where $\rho_r\colon \sll(2,\R)\to \gl(r,\R)$ is an irreducible $r$-dimensional representation of $\sll(2,\R)$, and $V(l-k)$ is an irreducible component of dimension $l-k+1$ in the decomposition of the tensor product of $\rho_k$ and $\rho_l$.

For example, in the simplest case of $(k,l)=(2,3)$ we have:
\[
\ag^{II}_{2,3}=
\left\{
\begin{pmatrix}
a+e_1 & c & p &  q & 0 \\
b  & -a+e_1 & 0 & p & q \\
0 & 0 & 2a+e_2 & 2c & 0\\
0 & 0 & b & e_2 & c \\
0 & 0 & 0 & 2b & -2a + e_2
\end{pmatrix}\right\}.
\]

If $k=l$, then 
\[
\ag^{II}_{k,k} = \left\{\left. 
\begin{pmatrix} \rho_k(A)+c_{11} E_{k} & c_{12} E_{k} \\ 
c_{21} E_{k} & \rho_k(A) + c_{22} E_{k} \end{pmatrix} 
\,\right|\, \begin{matrix} A\in \sll(2,\R), \\ c_{11},c_{12},c_{21},c_{22}\in\R \end{matrix}
\right\}.
\]

As $\gamma^{II}$ is dual to $\gamma^{I}$, the subalgebra $\ag^{I}_{k,l}$ is conjugate to the transposed of $\ag^{II}_{k,l}$. 

Now we can formulate the main result of the paper.
\begin{thm} Let $\ag^{I}_{k,l},\ag^{II}_{k,l}\subset \gl(V)$ be symmetry algebras of the curves $\gamma^{I}$ and $\gamma^{II}$ respectively. The Sternberg prolongations of $\ag^{I}_{k,l},\ag^{II}_{k,l}$ are finite dimensional and coincide with Lie algebras $\g^{I}_{k,l}$ and $\g^{II}_{k,l}$.
\end{thm}

The proof of this theorem is a direct corollary of the following technical result.
\begin{prop}
Let $\m$ be a graded nilpotent Lie algebra, $V$ a graded commutative ideal stable with respect to $\Der_0(\m)$, and let $E\subset \m$ be a commutative subalgebra in $\m_{-1}$ such that $\m=E\oplus V$. 

Assume that the action of $E$ on $V$ is faithful and denote by $\ad E$ the corresponding subalgebra in $\gl(V)$, concentrated in degree $-1$. Define $\ag$ as a largest graded subalgebra of $\gl(V)$ such that $\ag_{-}=\ad E$. Then the Tanaka prolongation of the pair $(\m, \Stab(E))$ coincides with the Sternberg prolongation of the subalgebra $\ag\subset \gl(V)$.
\end{prop}
\begin{proof}
Let $\g$ be the Sternberg prolongation of the subalgebra $\ag\subset\gl(V)$. We equip it with grading 
inherited from the grading of $V\subset\m$. Let us show that $\g$ satisfies conditions (1)-(2) of the Tanaka prolongation of the pair $(\m,\Stab(E))$ and hence is naturally embedded into $\g(\m,\Stab(E))$. 

It is clear that both $V$ and $\ad E \subset \ag$ are negatively graded and lie in $\g$. Suppose $u$ is a homogeneous element of $\g$ of negative degree. If it lies in $S^1(V^*)\otimes V = \gl(V)$, then by  construction of $\ag$ it is contained in $\ad E$. If $u\in S^k(V^*)\otimes V$ with $k\ge 2$, then we can find homogeneous elements $v_1,\dots,v_{k-1}\in V$ such that $\bar u=[\dots [u,v_1],\dots, v_{k-1}]$ is a non-zero element of $\ag$. As all $v_1,\dots,v_{k-1}\in V$ are negatively graded and $\deg \bar u\ge -1$, we see that $\deg u \ge 0$. Thus, negative part of $\g$ coincides with $(\ad E) \oplus V$ and is naturally isomorphic to $\m$.

Let $u\in\g$ be an arbitrary element of degree $0$. Let us show that it lies in $\ag_0$ and thus preserves $\ad E$. Another option would be that $u\in S^2(V^*)\otimes V$ and $[u,v]\in\ag$ for all $v\in V$. As $[u,v]$ is necessarily negatively graded, this would imply that $[u,v]\in\ad E$. Thus, $\ad u$ would define a non-zero degree preserving derivation of $\m$ that takes $V$ to $E$. But this contradicts the assumption that $V$ is stable with respect to $\Der_0(\m)$. This proves that $\g_0$ stabilizes $E$.

Now let $u\in\g$ be an arbitrary non-zero element of non-negative degree. Then by the property of Sternberg prolongation we have $[u,V]\ne 0$. This completes the proof that $\g$ is naturally embedded into the Tanaka prolongation of $(\m,\Stab(E))$.

Let us now prove that $\g(\m,\Stab(E))$ is embedded into the Sternberg prolongation of the subalgebra $\ag\subset \gl(V)$. As $V$ naturally lies in $\g(\m,\Stab(E))$ as a commutative subalgebra, we just need to prove that there is a complementary subalgebra $\g_0$ to $V$ such that each non-zero element in it has a non-zero bracket with $V$. We define $\g_0$ as $E+\sum_{i\ge 0} \g_i(\m,\Stab(E))$. Indeed, by definition of $\g(\m,\Stab(E))$ we have $[E,\g_0(\m,\Stab(E))]\subset E$ and $[E,\g_i(\m,\Stab(E))]\subset \g_{i-1}(\m,\Stab(E))$ for $i>0$. 

Next, let $u$ be an arbitrary non-zero element inside $\g_0$. Suppose $[u,V]=0$. This means that $u$ lies in the centralizer $Z(V)$ of $V$ in $\g(\m,\Stab(E))$. But as $[E,V]\subset V$, the centralizer $Z(V)$ is also stable with respect to the adjoint action of $E$. Hence, by the property (2) of Tanaka prolongation taking sufficiently many brackets of $u$ with elements from $E$, we get a non-zero element $\bar u$ from $Z_0(V)=Z(V)\cap \g_0(\m,\Stab(E))$. Moreover, this element acts trivially on $V$, but non-trivially on $E$. Thus, there is an element $e\in E$ such that $[\bar u, e]\ne 0$. But then for any element $v\in V$ we have:
\[
[[\bar u, e], v] = [\bar u, [e, v]] - [e, [\bar u, v]] =0
\]
This contradicts to the assumption that the action of $E$ on $V$ is faithful. 

So, for any non-zero element $u\in \g_0$ we have $[u,V]\ne 0$. This proves that $\g(\m,\Stab(E))$ coincides with the Sternberg prolongation of $\ag\subset \gl(V)$.
\end{proof}

\begin{ex}
According to the above theorem the Sternberg prolongation of the Lie algebra $\ag^{I}_{2,3}$ is equal to $\g^{I}_{2,3}$ and is isomorphic to $\gl(3,\R)\rtt S^2(\R^3)$. Let us describe the corresponding grading of $\g^{I}_{2,3}$:
\begin{align*}
&\deg -1: && \langle \dd{y}, \dd{z}, x\dd{y}, x\dd{z}, x^2\dd{z}\rangle;\\
&\deg 0: && \langle \dd{x}, x\dd{x}, x^2\dd{x}+xy\dd{y}+2xz\dd{z}, y\dd{y}, z\dd{z}, y\dd{z}, xy\dd{z}\rangle;\\
&\deg 1: && \langle y\dd{x}, y^2\dd{z}, xy\dd{x}+y^2\dd{y}+2yz\dd{z} \rangle.\\
\end{align*}
As expected, the degree $-1$ component is a commutative subalgebra.  

In general, in case of arbitrary $2\le k < l$ the grading of $\g^I_{k,l}$ according to Sternberg prolongation can be determined by setting $\deg x = 0$, $\deg y=\deg z = 1$ in formulas of Proposition~\ref{propI}. Note that the number of non-zero prolongations of $\ag^{I}_{k_l}$ can be arbitrarily high. For example, if $k=2$, then the $(l-2)$-nd prolongation is still non-zero, while $(l-1)$-st one already vanishes.
\end{ex}

\begin{ex}
The grading of $\g^{II}_{2,3}$ viewed as the Sternberg prolongation of $\ag^{II}_{2,3}$ can be determined by setting $\deg x = 0$, $\deg y=\deg z = \deg z_1=1$ in Proposition~\ref{propII}.

For $(k,l)\ne (2,3)$ the grading of $\g^{II}_{k,l}$ is defined by $\deg x = 0$, $\deg y=\deg z = 1$.
\end{ex}

\section{Non-linear mixed order equations}
\label{sec:nlin}
Non-linear mixed-order equations can be treated via the notion of $G$-structures on filtered manifolds introduced by N.~Tanaka. Consider a system:
\begin{equation}\label{nlin:kl}
\begin{aligned}
y^{(k)} &= f(x,y,y',\dots,y^{(k-1)},z,z',\dots,z^{(l-1)}),\\
z^{(l)} &= g(x,y,y',\dots,y^{(k-1)},z,z',\dots,z^{(l-1)}),
\end{aligned}
\end{equation}
where as above $2\le k<l$. It can be viewed as a codimension 2 submanifold $\E$ of the mixed jet space $J^{k,l}(\R,\R^2)$, which is transversal to the fibers of the projection $\pi\colon J^{k,l}(\R,\R^2) \to J^{k-1,l-1}(\R,\R^2)$. 

The canonical contact system on $J^{k,l}(\R,\R^2)$ restricted to the equation manifold $\E$ is given by the following 1-forms:
\begin{align*}
& dy_i - y_{i+1}dx, i=0,\dots k-2, \quad dz_j - z_{j+1}dx, j=0,\dots,l-2;\\
& dy_{k-1} - f(x,y_0,y_1,\dots,y_{k-1},z_0,z_1,\dots,z_{l-1}) dx;\\
& dz_{l-1} - g(x,y_0,y_1,\dots,y_{k-1},z_0,z_1,\dots,z_{l-1}) dx.\\
\end{align*}
It defines a one-dimensional vector distribution $E$, whose integral curves are lifts of solutions of the given system to the jet space $J^{k,l}(\R,\R^2)$.

As in the case of two different notions of symmetries, there are also two ways to introduce another vector distribution (or even a flag of distributions) transversal to $E$. Namely, the first way is to define a transversal foliation $F$ as the kernel of the projection $\pi_1\colon \E \to J^{0,0}(\R,\R)=\R^3$. The second way is to define $F$ as the kernel of the projection $\pi_2\colon \E\to J^{k-2,l-2}(\R,\R^2)$. 

Let us first consider the second case in more detail.  As in case of trivial ODEs, let $F$ be the 2-dimensional completely integrable distribution tangent to the fibers of the projection $\pi_2\colon \E\to J^{k-2,l-2}(\R,\R^2)$. Next, we define $D=E\oplus F$, which can also be defined as a pull-back of the standard contact system on $J^{k-1,l-1}(\R,\R^2)$. In particular, the weak derived flag of $D$ defines the filtration of the tangent bundle $T\E$ and turns $\E$ into a filtered manifold of type $\n^{II}_{k,l}$. The decomposition $D=E\oplus F$ is known as so-called \emph{pseudo-product structure} and introduced and first studied by N.~Tanaka~\cite{tan3}. As in case of trivial ODEs, we shall call this pseudo-product structure \emph{the EDS of second kind associated with a non-linear system of ODEs of mixed order}.

Its algebraic prolongation is already computed in Section~\ref{sec:tan} and is equal to $\g^{II}_{k,l}$. In particular, using the results of Tanaka~\cite{tan1} (see also Zelenko~\cite{zel:tan}) we immediately arrive at the following result.
\begin{prop}\label{propkind2}
For any EDS of second kind associated with system~\eqref{nlin:kl} there exists a natural frame bundle $P\to \E$ and an absolute parallelism structure $\omega\colon TP \to \g^{II}_{k,l}$.
\end{prop}

Let us now consider the EDS interpretation of~\eqref{nlin:kl} that contains the fibers of the projection $\pi_1\colon \E \to J^{0,0}(\R,\R)=\R^3$ as a part of its data. As in case of trivial systems, define $F_1$ as a distribution tangent to the fibers of the projection $\pi_1$:
\begin{equation}\label{F1def}
F_1 = \left \langle \dd{y_i}, i=1,\dots,k-1; \dd{z_j},j=1,\dots,l-1\right\rangle.
\end{equation}
Further, define a sequence of distributions:
\begin{equation}\label{Fidef}
F_{i+1} = \{ Y \in F_i \mid [Y, E]\subset F_i \}.
\end{equation}
We call a pair of distributions $(E,F_1)$ \emph{the EDS of first kind associated with a non-linear pair of ODEs of mixed order}. However, in general, the dimensions of these distributions in the non-trivial case may differ from the dimensions of these distributions for the trivial system of equations. To see this explicitly, consider the first non-trivial case of $(k,l)=(2,4)$ when this can be observed explicitly. In this case the vector distribution $E$ is spanned by:
\[
X = \dd{x} + y_1\dd{y_0} + f\dd{y_2} + z_1\dd{z_0} + z_2\dd{z_1} + z_3\dd{z_2} + g\dd{z_3},
\]
and $F_1$ is spanned by vector fields $\dd{y_1}$ and $\dd{z_i}$, $i=1,2,3$. Simple computation shows that, as expected, $F_2$ is spanned by $\dd{z_i}$, $i=2,3$. However, for $F_3$ we have already the branching:
\begin{align*}
F_3 &= 0,\qquad&\text{if } \frac{\partial f}{\partial z_3} \ne 0;\\
F_3 &= \left\langle \dd{z_3} \right\rangle, &\text{otherwise}.
\end{align*}

This example can be easily extended to arbitrary $k<l$. A straightforward generalization of Tanaka prolongation procedure~\cite{tan3,zel:tan} to filtered structures with constant non-fundamental symbol gives:
\begin{prop}\label{propkind1}
Let $\E$ be a non-linear pair of ODEs~\eqref{nlin:kl} such that $\frac{\partial f}{\partial z_i} = 0$, $i=k+1,\dots,l-1$. Then there exists a frame bundle $P\to \E$ and an absolute parallelism structure $\omega\colon TP \to \g^{I}_{k,l}$ naturally associated with the EDS of first kind for system~\eqref{nlin:kl}.
\end{prop}

\section{Flag structures}
\label{sec:flag}
Non-linear mixed order equations can be viewed as a particular case of another class of geometric structures. Let $M$ be an arbitrary smooth manifold of dimension $n$. Let $\alpha=(\alpha_1,\dots,\alpha_r)$ be any increasing sequence of integers $1\le \alpha_1 < \dots < \alpha_r < n$. Denote by $F_{\alpha}(T_pM)$, $p\in M$, the flag variety of subspaces in $T_pM$ of dimensions $\alpha_1,\dots,\alpha_r$ and by $F_{\alpha}(TM)$ or simply by $F_\alpha(M)$ the bundle of flag varieties at all points $p\in M$.

Let $V$ be a vector space of dimension $n$ and let $F_\alpha(V)$ be the flag variety of subspaces in $V$. To stick with the notational agreements of the Tanaka theory we number the subspaces in flags from  $F_{\alpha}(V)$ by negative integers, i.e. an element of $F_{\alpha}(V)$ is a decreasing by inclusion tuple $\{V_{-i}\}_{i=1}^r$ of subspaces of $V$ such that $\dim V_{-i}=\alpha_i$.  The flag variety $F_\alpha(V)$ is naturally equipped with a transitive action of the Lie group $GL(V)$. The bundle $F_\alpha(M)$ can also be viewed as bundle associated to the principal $GL(V)$-bundle $\mathcal{F}(M)$ of all frames in $M$.

For any curve $W(t)\subset V$ of subspaces of fixed dimension $m$ (i.e., a curve in a Grassmann variety $\Gr_m(V)$) we define $W'(t)$ as follows. Choose a family of smooth curves $v_1(t),\dots,v_m(t)$ in $V$ such that
\[
W(t) = \langle v_1(t),\dots,v_m(t) \rangle,\quad t\in \R.
\]
Then we define:
\[
W'(t) = W(t) + \langle v'_1(t),\dots,v'_m(t) \rangle,\quad t\in \R.
\]
It is easy to check that $W'(t)$ does not depend on the choice if the vectors $v_i(t)$. Moreover of $w(t)$ is a smooth curve in $V$ such that $w(t)\in W(t)$ then for any $t_0$ the image of the vector $w'(t_0)$ to the factor space $V/W(t_0)$ depends on the vector $w(t_0)$ only and not on the curve $w(t)$, i.e an element of $\Hom\bigl(W(t_0), V/W(t_0)\bigr)$ is assigned to the tangent vector to the curve $t\mapsto W(t)$ at $t_0$.

Let now $\Gamma\subset F_\alpha(V)$ be an unparametrized curve. Fixing any local parameter $t$ on $\Gamma$, we can explicitly write it as:
\begin{equation}\label{flagcurve}
0 \subset V_{r}(t) \subset V_{r-1}(t)\subset\dots \subset V_{1}(t)\subset V_0 \subset V.
\end{equation}
We say that the curve $\Gamma$ is \emph{integral} or \emph{compatible with respect to differentiation}, if it $V_i'(t)\subset V_{i-1}(t)$ for all $i=1,\dots,r$. Similarly to the previous paragraph, the tangent line to the curve $\Gamma$ at the point corresponding to a parameter $t$ can be associated with a line generated by a degree $-1$ endomorphisms $X_t$ of the graded space $\displaystyle{\bigoplus_{i=0}^{r} V_i(t)/V_{i+1}(t)}$, where $V_{r+1}(t):=0$. Identifying all graded spaces $\displaystyle{\bigoplus_{i=0}^{r} V_i(t)/V_{i+1}(t)}$ with one grading of the space $V$  we can consider the line $\langle X_t\rangle$ of degree $-1$ endomorphisms of $V$, defined up to the conjugations by linear isomorphisms preserving the grading. This line (or more precisely its equivalence class with respect to the above conjugation) is called the \emph{flag symbol of the curve \eqref{flagcurve} at the point $t$}. The curve of flags is said to be of constant type $\langle X\rangle$ if its flag symbols at any point belong to the same equivalence class of the line $\langle X\rangle$ (with respect to  the conjugations by linear isomorphisms preserving the grading). Note that absolutely the same constructions can be done if the indices in \eqref{flagcurve} are shifted somehow.

\begin{dfn}
A \emph{flag structure on a manifold $M$} is a smooth one-dimensional subbundle $\mathcal C$ of the flag bundle $F_\alpha(TM)$ such that $\mathcal C_p\subset F_\alpha(T_pM)$ is an integral curve for all $p\in M$. We say that a flag structure has a constant flag symbol $\langle X\rangle$ if all its fibers are curves of flags with constant flag symbol $\langle X\rangle$.
\end{dfn}

Flag structures, and even in more general setting, were studied in our recent preprint \cite{quasi}.
Assume that, like in Proposition \ref{flagprolong}, $\ag(\langle X\rangle)$ is the largest graded subalgebra of $\gl(V)$ whose negative part is one-dimensional and is generated by $X$. The algebra $\ag(\langle X\rangle)$ is called the \emph{universal prolongation  of the flag symbol $\langle X\rangle$ in $\gl(V)$.}  Let $\g(\langle X\rangle)$ be the Sternberg prolongation of the algebra $\langle X\rangle$, as described in section~5. The following Theorem is a particular case of Theorem 2.4 in \cite{quasi}:

\begin{thm}
\label{flagthm}
For any flag structure on a manifold $M$  with constant flag symbol $\langle X\rangle$  there exists a natural frame bundle $P\to M$ and an absolute parallelism structure $\omega\colon TP \to \g(\langle X\rangle)$.
\end{thm}

Now we show how to assign a natural flag structure with constant symbol to the EDS, both of first and second kind,  associated with a systems of ODEs \eqref{nlin:kl} on the space of solutions of this system.
Let  $\Fol(\mathcal E)$ be the foliation of the equation submanifold $\mathcal E$ the (prolonged) solutions of the system  \eqref{nlin:kl} or, equivalently, by the integral curves of the rank $1$ distribution $E$. Then Propositions \ref{propkind2} and \ref{propkind1} can be seen as particular cases of Theorem \ref{flagthm}.

Indeed, let us pass to the quotient manifold of $\mathcal E$  by the foliation $\Fol(\mathcal E)$, i.e. to the space of solutions of \eqref{nlin:kl}. Locally we can assume that there exists a
quotient manifold 
\[
\text{Sol}(\mathcal E)=\mathcal  E/\Fol(E),
\]
whose points are leaves of $\Fol(E)$ or, equivalently, (prolonged) solutions of  \eqref{nlin:kl}.
Let $\Phi\colon \mathcal E\to \Fol(E)$ be the canonical projection to the quotient manifold. 

Fix a leaf $\gamma$ of  $\Fol(E)$. For the EDS of the second kind, as in section 6,  let $F$ be the 2-dimensional completely integrable distribution tangent to the fibers of the projection $\pi_2\colon \E\to J^{k-2,l-2}(\R,\R^2)$:
\begin{equation}
\label{Ji}
V_{-1}(x):=\Phi_*\bigl(F(x)\bigr), \quad x\in \gamma.
\end{equation}
Then the curve $x\mapsto V_{-1}(x), x\in\gamma$ is a curve in the Grassmannian of planes in $T_\gamma \Fol(E)$, Taking differentiation, one can generate  from this curve the following curve  of flags in $T_\gamma \Fol(E)$:
\begin{equation}
C^{II}_\gamma=\{x\mapsto \{0 \subset V_{l-1}(x) \subset V_{l-2}\dots \subset V_{0}(x)= T_\gamma \Fol(E)\}: x\in\gamma\},
\end{equation}
where the spaces $V_{i}(x)$ are defined inductively: $V_{i-1}(x):=V_{i}'(x)$. Note that each curve $C^{II}_\gamma$ is a curve with constant flag symbol, generated by $X$ as in \eqref{Xdef} with the grading of the second kind  (see the paragraph after the formula \eqref{Xdef}).

Then the bundle $C^{II}\rightarrow \Fol(E)$ is the flag structure on  $\Fol(E)$ with the constant flag symbol generated by this $X$ and the problem of equivalence of the EDSs of second kind is the same as the problem of equivalence of the corresponding  flag structures $C^{II}$. Proposition \ref{propkind2} is the consequence of Proposition \ref{flagprolong} and Theorem \ref{flagthm}.

In the same way for the EDS of the first kind, if $F_i$ with $1\leq i\leq l-1$  are the distributions defined by relations \eqref{F1def}-\eqref{Fidef} and 
$F_0=T \Fol(E)$, then let 
\begin{equation}
C^{I}_\gamma=\{x\mapsto\{\Phi_*\bigl(F_i(x)\bigr)\}_{i=0}^{l-1}: x\in\gamma\},
\end{equation}
Note that by the definition the curve of flags $C^{I}_\gamma$ is compatible with respect to the differentiation. Moreover, in case of EDS of the first kind associated with ODE \eqref{nlin:kl} satisfying $\frac{\partial f}{\partial z_i} = 0$, $i=k+1,\dots,l-1$ the curve $C^{I}_\gamma$ is exactly a curve with constant flag symbol, generated by $X$ as in \eqref{Xdef} with the grading of the first kind (see the paragraph after the formula~\eqref{Xdef}). Then for such EDS the bundle $C^{I}\rightarrow \Fol(E)$ is the flag structure on  $\Fol(E)$ with the constant flag symbol generated by this $X$ and the problem of equivalence of the EDSs of first kind is the same as the problem of equivalence of the corresponding  flag structures $C^{I}$. Proposition \ref{propkind1} is the consequence of Proposition \ref{flagprolong} and Theorem \ref{flagthm}.

\section{Other kinds of EDS associated with an ODE system of mixed order}
\label{sec:shifts}

As shown above, the two different EDS interpretations of a given trivial or general system of two ODEs of mixed order come from fixing different jet space projections.

Namely, in case of an EDS of first kind we assume that the projection $J^{k,l}(\R,\R^2)\to J^{0,0}(\R,\R^2)$ is preserved. Hence, all symmetries of the first kind are just point transformations that preserve equation $\E\subset J^{k,l}(\R,\R^2)$. Moreover, due to the classical Lie theorem any transformation preserving the contact system on $J^{r,r}(\R,\R^2)$, $r\ge 0$, also preserves the projection $J^{r,r}(\R,\R^2)\to J^{0,0}(\R,\R^2)$. Therefore, defining an EDS of first kind we could as well assume that any of the projections $J^{k,l}(\R,\R^2)\to J^{r,r}(\R,\R^2)$, $r=0,\dots,k$, is preserved. 

In the same manner, an EDS of second kind is defined by assuming that the projection $J^{k,l}(\R,\R^2)\to J^{k-1,l-1}(\R,\R^2)$ is preserved. This automatically implies that any of the projections $J^{k,l}(\R,\R^2)\to J^{k-r,l-r}(\R,\R^2)$, $r=0,\dots,k$ is also preserved. 

But we can define other kinds of EDS assuming that we preserve any of the projections $J^{k,l}(\R,\R^2)\to J^{k-r,l-s}(\R,\R^2)$ for arbitrary $r=1,\dots,k$, $s=1,\dots,l$. Similar to Lie theorem, we can prove that this automatically implies that the projections $J^{k,l}(\R,\R^2)\to J^{k-r-1,l-s-1}(\R,\R^2)$ and $J^{k,l}(\R,\R^2)\to J^{k-r+1,l-s+1}(\R,\R^2)$ are also preserved, assuming that all indexes are non-negative. Thus, only the difference $\delta=(l-s)-(k-r)$ is important. 

In more detail, we have:
\begin{dfn} An \emph{EDS of shift $\delta$} associated with a trivial system of equations $y^{(k)}=z^{(l)}=0$, $k\le l$ is defined as a pair of vector distributions on the equation $\E={y_k=z_l=0}\subset J^{k,l}$:
\begin{align*}
E &= \langle \dd{x}+y_1\dd{y_0} + \dots y_{k-1} \dd{y_{k-2}} + z_1\dd{z_0} + \dots z_{l-1} \dd{z_{l-2}} \rangle,\\
F &= \langle \dd{y_1}, \dots, \dd{y_{k-1}}, \dd{z_{\delta+1}}, \dots, \dd{z_{l-1}}\rangle,
\end{align*}
if $\delta \ge 0$, and:
\begin{align*}
E &= \langle \dd{x}+y_1\dd{y_0} + \dots y_{k-1} \dd{y_{k-2}} + z_1\dd{z_0} + \dots z_{l-1} \dd{z_{l-2}} \rangle,\\
F &= \langle \dd{y_{-\delta+1}}, \dots, \dd{y_{k-1}}, \dd{z_1}, \dots, \dd{z_{l-1}}\rangle,
\end{align*}
if $\delta < 0$.
\end{dfn}
Note that  this definition includes also the case of $k=l$, where, as we shall see below, different values of the shift lead to different symmetry algebras. Although the definition makes sense for arbitrary values of $\delta$, we shall be mainly interested in non-trivial cases, when $\delta$ is in the range from $-k+2$ to $l-2$.

As above, the vector distribution $E$ is tangent to the solutions of the system, while the distribution $F$ is tangent to the fibers of the projection $\pi\colon \E\to J^{0,\delta}$. In particular, we see that the EDS of first kind corresponds to the shift $\delta=0$, while the EDS of second kind corresponds to the shift $\delta=l-k$.

There is a graphical way to encode an arbitrary EDS of shift $\delta$ using a skew Young tableau (i.e., the tableau does not need to be aligned to the left). Namely, the tableau consists of two rows with $l$ and $k$ boxes respectively, the rows are right aligned for $\delta=0$, the row with $l$ cells is shifted by extra $\delta$ boxes to the right if $\delta$ is positive or to $-\delta$ cells to the left otherwise. The restriction $-k+2 \le \delta \le l-2$ means that the rows overlap at least in two cells. 

Such graphical notation provides an easy way to describe the grading on the vector space $V$ and degree $-1$ operator $X\in \gl(V)$ as defined in Section~\ref{sec:sym}. Namely, each box in the tableau corresponds to a basis element in $V$. The grading is defined in such way that it decreases by $1$ from left to right starting from $-1$ with basis elements in the same column having the same degree. And the operator $X$ maps each basis element to another basis element corresponding to the box on the right, or to $0$, if there is no right neighbor. The corresponding graded nilpotent Lie algebra $\m$ is defined as $\R X\oplus V$. Below we list a number of examples of such tableaux. 

Tableaux and preserved jet space projections for systems of mixed order $(2,3)$:
\begin{align*}
\delta = 1: &\quad\young(\ \ \ ,\ \ )\quad J^{2,3}(\R,\R^2)\to   J^{1,2}(\R,\R^2) \to J^{0,1}(\R,\R^2);\\[2mm]
\delta = 0: &\quad\young(\ \ \ ,:\ \ )\quad J^{2,3}(\R,\R^2)\to   J^{1,1}(\R,\R^2) \to J^{0,0}(\R,\R^2).
\end{align*}

Tableaux and preserved jet space projections for systems of mixed order $(2,4)$:
\begin{align*}
\delta = 2: &\quad \young(\ \ \ \ ,\ \ )\quad J^{2,4}(\R,\R^2)\to   J^{1,3}(\R,\R^2) \to J^{0,2}(\R,\R^2);\\[2mm]
\delta = 1:  &\quad \young(\ \ \ \ ,:\ \ )\quad J^{2,4}(\R,\R^2)\to   J^{1,2}(\R,\R^2) \to J^{0,1}(\R,\R^2);\\[2mm]
\delta = 0:  &\quad \young(\ \ \ \ ,::\ \ )\quad J^{2,4}(\R,\R^2)\to   J^{1,1}(\R,\R^2) \to J^{0,0}(\R,\R^2).
\end{align*}

Tableaux and preserved jet space projections for systems of mixed order $(3,4)$:
\begin{align*}
\delta = 2: &\quad \young(:\ \ \ \ ,\ \ \ )\quad J^{3,4}(\R,\R^2)\to   J^{1,3}(\R,\R^2) \to J^{0,2}(\R,\R^2);\\[2mm]
\delta = 1: &\quad \young(\ \ \ \ ,\ \ \ )\quad J^{3,4}(\R,\R^2)\to   J^{2,3}(\R,\R^2) \to J^{1,2}(\R,\R^2)\to J^{0,1}(\R,\R^2);\\[2mm]
\delta = 0: &\quad \young(\ \ \ \ ,:\ \ \ )\quad J^{3,4}(\R,\R^2)\to   J^{2,2}(\R,\R^2) \to J^{1,1}(\R,\R^2)\to J^{0,0}(\R,\R^2);\\[2mm]
\delta = -1: &\quad \young(\ \ \ \ ,::\ \ \ )\quad J^{3,4}(\R,\R^2)\to   J^{2,1}(\R,\R^2) \to J^{1,0}(\R,\R^2).
\end{align*}

In fact, such tableaux can be used to encode different EDS associated with mixed order systems of ODEs with an arbitrary number of equations. As this topic lies outside of the scope of the current paper, we just show one particular example corresponding to an EDS associated with a system of order $(2,3,4)$:
\[
\young(\ \ \ \ ,::\ \ \ ,::\ \ )\quad J^{2,3,4}(\R,\R^3)\mapsto J^{1,2,1}(\R,\R^3)\mapsto J^{0,1,0}(\R,\R^3).
\]

Using the results from Section~\ref{sec:sym}, we can easily compute the symmetry algebra of an EDS of shift $\delta$ associated with a system~\eqref{odekl}. All other results including the relationship between Sternberg and Tanaka prolongations and the existence of natural frame bundles in case of non-linear systems are also easily generalized to the EDSs of arbitrary shift $\delta$. 

\begin{prop}
The symmetry algebra of an EDS of shift $\delta$ associated with a system~$y^{(k)}=z^{(l)}=0$ is described as follows.

1. If $\delta$ satisfies $0<\delta <l-k$ (i.e., when one of the rows in the corresponding skew Young tableau lies strictly within the other row), then the symmetry algebra is generated by prolongations of vector fields:
\begin{align*}
	&\dd{x}, x\dd{x}, x^2\dd{x} + (k-1)xy_0\dd{y_0} + (l-1)xz_0\dd{z_0},\\
         & y_0\dd{y_0}, z_0\dd{z_0},\\
	& x^i \dd{y_0},x^j\dd{z_0}\quad i=0,\dots,k-1;\ j=0\dots,l-1.
\end{align*}
The cases $\delta = 0$ and $\delta = l-k$ are covered by Propositions~\ref{propI} and~\ref{propII} respectively.

2. If $3=k\le l$ and $\delta=-1$, the symmetry algebra is generated by prolongations of the following vector fields:
\begin{small}
\begin{align*}
& x(2y_0-xy_1)\dd{x}+(2y_0^2-\frac{1}{2}x^2y_1^2)\dd{y_0}+y_1(2y_0-xy_1)\dd{y_1} + (l-1)z_0(2y_0-xy_1)\dd{z_0},\\
& 2(y_0-xy_1)\dd{x}-xy_1^2\dd{y_0}-y_1^2\dd{y_1}-(l-1)z_0y_1\dd{z_0},\\
& x^2\dd{x}+2xy_0\dd{y_0}+2y_0\dd{y_1}+(l-1)xz_0\dd{z_0},\\
&x\dd{x}-y_1\dd{y_1},\ y_0\dd{y_0} +y_1\dd{y_1},\ z_0\dd{z_0},\\
& y_1\dd{x} + \frac{1}{2}y_1^2\dd{y_0},\ \dd{x},\ \dd{y_0},x\dd{y_0}+\dd{y_1},x^2\dd{y_0}+2x\dd{y_1},\\
&x^i(xy_1-2y_0)^{j_0} y_1^{j_1}\dd{z_0},\quad i+j_0+j_1\le l-1.
\end{align*}
\end{small}
This symmetry algebra is isomorphic to $\mathfrak{csp}(4,\R)\rtt S^{l-1}(\R^4)$. 

4. If $3\le k\le l$ and $\delta > l-k$, then the symmetry algebra is generated by the prolongations of the following vector fields:
\begin{align*}
	&\dd{x}, x\dd{x}, x^2\dd{x} + (k-1)xy_0\dd{y_0} + (l-1)xz_0\dd{z_0},\\
         & y_0\dd{y_0}, z_0\dd{z_0}, x^i \dd{z_0},\quad i=0,\dots,l-1;\\
	& x^i g_{\delta,l-\delta}^{(s_1)} \dots g_{\delta,l-\delta}^{(s_j)}\dd{y_0}, \quad j\ge 0, i+ (l-\delta-1) j\le k-1,
\end{align*}
where the functions $g_{\delta,l-\delta}^{(s)}$ are defined in Section~\ref{sec:sym}.

5. The symmetry algebra in the case $4 \le k\le l$ and $\delta<0$ is obtained from the previous item by exchanging $y$ and $z$, $k$ and $l$ and replacing $\delta$ by $-\delta$.
\end{prop}
\begin{proof}
The proof is similar to the proof of Proposition~\ref{propII} and is reduced to solving equations~\eqref{eq:dkl} for $p=k-1$ and $q=l-\delta-2$. 
\end{proof}

This result can be considered as a direct generalization of Propositions~\ref{propI} and~\ref{propII}. The results of Sections~\ref{sec:tan} and~\ref{sec:g} also generalize to the EDS systems of arbitrary shift. In particular, all symmetry algebras from the above proposition can be obtained as Tanaka prolongations of the corresponding (possibly, non-fundamental) graded nilpotent Lie algebras as well as Sternberg prolongations of certain subalgebras in $\gl(k+l,\R)$.

The existence of natural frame bundles in case of non-linear systems and the relation to geometry flag structures are also easily generalized. See~\cite{quasi} for more detail. In particular, the equivalence problem for EDS of arbitrary shift $\delta$ coincides with $(k-l-\delta)$-equivalence of non-linear systems of ODEs of mixed order $(k,l)$ introduced in~\cite[Section 3]{quasi}.

\end{document}